\documentclass[12pt,reqno]{amsart}

\usepackage{amssymb}
\usepackage[all]{xy}

\newtheorem{theorem}{Theorem}[section]
\newtheorem{proposition}[theorem]{Proposition}

\newtheorem{corollary}[theorem]{Corollary}

\theoremstyle{definition}
\newtheorem{definition}[theorem]{Definition}

\numberwithin{equation}{section}

\begin{document}

\baselineskip=15.5pt

\title[Holomorphic geometric structures on Oeljeklaus--Toma manifolds]{Holomorphic
geometric structures on Oeljeklaus--Toma manifolds}

\author[I. Biswas]{Indranil Biswas}

\address{Department of Mathematics, Shiv Nadar University, NH91, Tehsil Dadri,
Greater Noida, Uttar Pradesh 201314, India}

\email{indranil.biswas@snu.edu.in, indranil29@gmail.com}

\author[S. Dumitrescu]{Sorin Dumitrescu}

\address{Universit\'e C\^ote d'Azur, CNRS, LJAD, France}

\email{dumitres@unice.fr}



\date{}

\begin{abstract}
We prove that any holomorphic geometric structure of affine type on an Oeljeklaus--Toma manifold is
locally homogeneous. For locally conformally K\"ahler Oeljeklaus--Toma manifolds we prove that all holomorphic
geometric structures, and also all holomorphic Cartan geometries, on them are locally homogeneous.
\end{abstract}

\maketitle

\section{Introduction}

The Oeljeklaus--Toma manifolds (which are simply referred to as $\text{OT}$--manifolds) form an important class of compact complex non-K\"ahler manifolds. 
They were constructed by Oeljeklaus and Toma in \cite{OT} using tools from algebraic number theory and generalizing the construction 
of Inoue surfaces in class ${\rm VII}_0$ \cite{In}. Oeljeklaus--Toma manifolds provide a counterexample to Vaisman's conjecture 
\cite{Va2}, which is a conjecture on locally conformally K\"ahler geometry that was open for 25 years.

The geometry of Oeljeklaus--Toma manifolds was systematically studied in the last two decades (see, for instance, \cite{APV,BO, Ka, 
Ka2,OV, OVV, PV, Ve}). In particular, it was proved that $\text{OT}$--manifolds share many of the geometric features of Inoue surfaces 
of class ${\rm VII}_0$ : they are non-K\"ahler flat affine manifolds \cite{OT}, they are of algebraic dimension zero (meaning that 
all meromorphic functions are constants) \cite{BO}, they contain no holomorphic curves \cite{Ve} and they are known to be 
solvmanifolds admitting no Vaisman metrics \cite{Ka2}. Moreover, locally conformally K\"ahler $\text{OT}$--manifolds do not admit proper 
complex subvarieties \cite{OV}.

Let us also mention that, among other results, it was proved in \cite{OT} that the canonical bundle of an $\text{OT}$--manifold and 
its tensor powers are flat. This result was extended in \cite{APV} to all line bundles.

Our article studies holomorphic geometric structures, and also holomorphic Cartan geometries, on Oeljeklaus--Toma manifolds.

By construction, any $\text{OT}$--manifold $M$ admits a standard complex affine structure \cite{OT}, which endows $M$ with a
torsion-free holomorphic flat affine connection \cite{Ko}. This is an example of holomorphic geometric structure in Gromov's sense \cite{Gr, 
DG} and can also be seen as a flat holomorphic Cartan geometry \cite{Sh}.
 
Our main results in this context are the following:
 
\begin{theorem}\label{thm 1 introduction} 
Let $M$ be an $\text{OT}$--manifold, and let $\phi$ be a holomorphic geometric structure
of affine type on $M$. Then $\phi$ is locally homogeneous. Moreover, the Lie algebra of local holomorphic affine vector fields on $M$
preserving $\phi$ is transitive.
\end{theorem}

In particular, the above theorem asserts that any holomorphic tensor on an $\text{OT}$--manifold is locally homogeneous. Moreover, the 
Lie algebra of all local holomorphic affine vector fields preserving the holomorphic tensor is transitive on $M$.

\begin{theorem} \label{thm 2 introduction}
Let $M$ be a locally conformally K\"ahler $\text{OT}$--manifold. Then the following statements hold:
\begin{enumerate}
\item[{(i)}] Any holomorphic geometric structure $\phi$ on $M$ is 
locally homogeneous. Moreover, the Lie algebra of all local holomorphic affine vector fields on $M$
that preserve the geometric structure $\phi$ is transitive.

\item[{(ii)}] Any holomorphic Cartan geometry of algebraic type on $M$ is 
locally homogeneous. Moreover, the Lie algebra of local holomorphic affine vector fields on $M$
preserving the Cartan geometry is transitive.
\end{enumerate}
\end{theorem}

The article is organized as follows. Section \ref{sect geom struct} introduces the two concepts of geometric structures which are 
studied in the paper: geometric structures in Gromov's sense \cite{Gr, DG} and Cartan geometries \cite{Sh}. Section \ref{sect OT} 
briefly presents the construction of $\text{OT}$--manifolds as discovered in \cite{OT} and also recall some of their main geometric features with 
focus on the properties used in our proofs. Section \ref{sect struct OT} provides the proofs of our results, in particular that of 
Theorem \ref{thm 1 introduction} (see Theorem \ref{thm1}) and of Theorem \ref{thm 2 introduction} (see Theorem \ref{lck OT} and 
Corollary \ref{Cartan lck OT}).

\section{Holomorphic geometric structures}\label{sect geom struct}

Let $M$ be a complex manifold of complex dimension $n$. For any integer $k \,\geq \,1$, we associate the 
principal bundle on it of $k$-frames $$R^k(M) \,\longrightarrow\, M,$$ which is the bundle of $k$-jets of 
local holomorphic coordinates on $M$. The structure group of this principal bundle is $D^k$: the group of all $k$-jets of 
local biholomorphisms of $\mathbb{C}^n$ fixing the origin. It may be mentioned that $D^k$ is a complex affine 
algebraic group \cite{Gr, DG,Ho,Hu}.

We define a geometric structure on $M$ as in \cite{Gr} (see also the expository survey \cite{DG}). 
\begin{definition}\label{def-a} A {\it holomorphic geometric structure} of order $k$ on $M$ is a holomorphic 
$D^k$-equivariant map $\phi$ from $R^k(M)$ to a complex algebraic manifold $Z$ endowed with an algebraic 
action of the group $D^k$.

Such a geometric structure $\phi$ is said to be {\it of affine type} if the above algebraic manifold $Z$ is a complex affine variety. 
\end{definition}

Holomorphic tensors are holomorphic geometric structures of affine type of order one. Indeed, a holomorphic 
tensor on $M$ is a holomorphic ${\rm GL}(n, \mathbb C)$-equivariant map from the frame bundle $R^1(M)$ to a 
linear complex algebraic representation $W$ of ${\rm GL}(n, \mathbb C)$. More precisely, for any integers 
$p,\,q \,\geq\, 0$, a holomorphic section of $(TM)^{\otimes p}\otimes (T^*M)^{\otimes q}$ is a ${\rm GL}(n, 
\mathbb C)$-equivariant holomorphic map from $R^1(M)$ to the representation $(\mathbb C)^{\otimes p} \otimes 
(\mathbb C^*)^{\otimes q}$.

Holomorphic affine connections are holomorphic geometric structures of affine type of order two \cite{Gr,DG}. 
Holomorphic foliations and holomorphic projective connections are holomorphic geometric structure of 
non-affine type.

The natural notion of (local) symmetry of a holomorphic geometric structure is classically defined as follows.

A (local) biholomorphism of $M$ preserves a holomorphic geometric structure $\phi$ if its canonical lift to 
the associated $k$-frame $R^k(M)$ fixes each fiber of the map $\phi$. Such a local biholomorphism is called a 
{\it local isometry} of $\phi$.

A (local) holomorphic vector field on $M$ is called a {\it Killing vector field} with 
respect to $\phi$ if its (local) flow acts on $M$ by local isometries.

\begin{definition}\label{d-rigid}
A holomorphic geometric structure $\phi$ is called {\it rigid} of order $l$ in 
Gromov's sense if any local biholomorphism preserving $\phi$ is uniquely determined by its 
$l$-jet at any given point.
\end{definition}

Holomorphic affine connections are rigid of order one in Gromov's sense (see \cite{Gr} 
and \cite{DG}). The rigidity comes from the fact that local 
biholomorphisms fixing a point and preserving a connection linearize in exponential 
coordinates, so they are indeed completely determined by their differential at the fixed 
point.

Holomorphic Riemannian metrics, holomorphic projective connections and holomorphic 
conformal structures of dimension at least three are all rigid holomorphic geometric structures,
while holomorphic symplectic structures and holomorphic foliations are non-rigid geometric 
structures~\cite{DG}.

By a theorem of Amores \cite{Am}, analytic local Killing fields on simply connected manifolds extend to global 
holomorphic Killing fields. This implies that the sheaf of local Killing fields of a holomorphic rigid 
geometric structure $\phi$ is locally constant. Its fiber is a finite dimensional Lie algebra called {\it the 
Killing Lie algebra} of $\phi$ \cite{DG,Gr}. If the Killing Lie algebra is transitive on an open dense subset $U$ 
in $M$, we say that that $\phi$ is {\it locally homogeneous} on $U$.

Let us now recall the classical definition of a {\it holomorphic Cartan geometry}.

Let $G$ be a connected complex Lie group and $H\, \subset\, G$ a connected
complex Lie subgroup. We denote by $\mathfrak h$ and $\mathfrak g$ the Lie algebras of $H$ and $G$
respectively.

\begin{definition}\label{def}
A holomorphic Cartan geometry of type $(G,\,H)$ on a {complex} manifold $M$ is a holomorphic principal 
$H$--bundle $\pi \,:\,\, E_H\, \longrightarrow\, M$ endowed with a $\mathfrak g$--valued holomorphic $1$--form 
$\eta$ on $E_H$ satisfying the following conditions:
\begin{enumerate}
\item $\eta\,\, :\,\, TE_H\,\, \longrightarrow\,\, E_H\times{\mathfrak g}$ is an
isomorphism;

\item $\eta$ is $H$--equivariant with $H$ acting on $\mathfrak g$ via the adjoint representation;

\item The restriction of $\eta$ to each fiber of $\pi$ coincides with the Maurer--Cartan form
associated to the action of $H$ on $E_H$.
\end{enumerate}
\end{definition}

The above definition is the infinitesimal version of the canonical fibration $G \,
\longrightarrow\, G/H$, seen as a holomorphic principal $H$-bundle equipped with the
left-invariant Maurer-Cartan form of $G$ (see \cite{Sh}).

The Cartan geometry is called {\it flat} if its curvature $\Omega(\eta)\,=\, d \eta+ \frac{1}{2} 
\lbrack \eta,\, \eta \rbrack$ vanishes identically.

By a classical result of Cartan, a flat Cartan geometry of type $(G,\,H)$ given by a principal
$H$--bundle $E_H$ and a form 
$\eta$ is flat if and only if $(E_H,\, \eta)$ is locally isomorphic to the fibration $G \,\longrightarrow\, G/H$ 
endowed with its canonical Maurer-Cartan form. As a consequence, Ehresmann proved that a complex manifold $M$ 
admits a flat Cartan geometry of type $(G,\,H)$ if and only if $M$ admits a holomorphic $(G,\, G/H)$--structure, 
which is a holomorphic atlas with values in the model space $G/H$ with the transition maps being given by 
the restriction of left-translation action, on $G/H$, of elements of $G$.

The Cartan geometry $(E_H,\, \eta)$ of type $(G,\,H)$ will be called of {\it algebraic type} if the image of $H$ 
through the adjoint representation of $G$ is an algebraic subgroup of ${\rm Aut}(\mathfrak{g})$ \cite{Ho,Hu}.

Holomorphic affine and projective connections, holomorphic conformal structures and holomorphic Riemannian 
metrics can be defined as holomorphic {\it Cartan geometries of algebraic type} \cite{Sh}. In particular, for 
$G\,=\, {\rm GL}(n , \mathbb C) \ltimes {\mathbb C}^n$ and $H\,=\, {\rm GL}(n , \mathbb C) $ one gets the 
definition of a Cartan geometry with model the complex affine space. In this case the bundle $E_H$ is 
necessarily isomorphic to the bundle of $1$-frames $R^1(M)$. This data is equivalent to the existence of a 
holomorphic affine connection $\nabla$ on $M$. The Cartan curvature is flat if and only if $\nabla$ is 
torsion-free and its curvature vanishes. By Ehresmann result, this is equivalent to the data of a holomorphic 
$(G\,=\, {\rm GL}(n , \mathbb C) \ltimes {\mathbb C}^n , {\mathbb C}^n)$--structure on $M$. Such an equivalence 
class of atlases with transition maps which are restrictions of affine transformations is classically termed 
as a {\it complex affine structure}.

There is a natural notion of (local) symmetry for holomorphic Cartan geometries. A (local) Killing field of the 
Cartan geometry $(E_H, \,\eta)$ is a local holomorphic vector field on $M$ which lifts to a local holomorphic 
vector field on $E_H$ whose local flow acts by bundle automorphisms on $E_H$ that preserve $\eta$. A local 
Killing field $X$ of a Cartan geometry extends to any simply connected open subset in $M$ containing
the domain of definition 
of $X$ \cite{Am}. This implies that the sheaf of local Killing vector fields is a locally constant sheaf on 
$M$ with fiber a finite dimensional Lie algebra; this Lie algebra
is called the {\it Killing Lie algebra} of the Cartan 
geometry. If the Killing Lie algebra is transitive on an open dense subset $U$ in $M$, we say that $(E_H,\, \eta)$ 
is locally homogeneous on $U$.

A flat Cartan geometry with model $(G,\,H)$ is locally homogeneous on $M$; its Killing Lie algebra is the Lie 
algebra $\mathfrak g$ of $G$. In particular, a complex affine structure on $M$ is locally homogeneous, and its 
local Killing Lie algebra is that of holomorphic affine vector fields (isomorphic to the Lie algebra of the 
complex affine group $G\,=\, {\rm GL}(n , \mathbb C) \ltimes {\mathbb C}^n$).

We will see in Section \ref{sect OT} that all $\text{OT}$--manifolds admit, by construction, a standard complex affine structure.

\section{Oeljeklaus--Toma manifolds}\label{sect OT}

An important new class of non-K\"ahler complex manifolds, generalizing Inoue surfaces in class ${\rm VII}_0$ \cite{In}, was 
constructed by Oeljeklaus and Toma, \cite{OT}, using ingredients from algebraic number theory. In this 
section we briefly describe the construction of those Oeljeklaus--Toma manifolds (which will be also denoted 
as $\text{OT}$--manifold) highlighting their geometric properties which will be useful for our study.

Let us follow \cite{OT} (see also \cite{PV}) and consider an algebraic number field $F$ of degree $n$ over 
$\mathbb Q$. We denote by ${\mathcal O}_F$ the ring of algebraic integers of $F$, and by ${\mathcal O}^{*}_F$ 
its multiplicative group of units. We assume that $F$ admits $n\,=\,s+2t$ embeddings into $\mathbb C$, with $s
\,>\,0$ of 
them being real and $2t \,>\,0$ being complex non-real. Let us denote by $\sigma_i$, with $i \,\in\,
\{1,\, \cdots,\, s\}$, the real embeddings and by $\sigma_{s+i}\,=\,\overline{\sigma}_{s+i+t}$,
where $i \,\in\, \{1,\, \cdots,\, t\}$, the complex non-real ones.

Using the above embeddings one defines a natural map
$$\sigma\,\, :\,\, F \,\,\longrightarrow\,\, {\mathbb C}^{s+t},\ \ \,
x\,\longmapsto\,(\sigma_1(x),\, \cdots, \,\sigma_{s+t}(x)).$$
The image of the algebraic integers $\mathcal O_F$ through $\sigma$ is known to be a lattice
of rank $n\,=\,s+2t$ in ${\mathbb C}^{s+t}$. Therefore there is a natural action of
$\mathcal O_F$ on ${\mathbb C}^{s+t}$ by translations, defined as
$$T(a,\, (z_1,\, \cdots ,\,z_{s+t}))\,=\,(\sigma_1(a)+z_1,\, \cdots,\, \sigma_{s+t}(a) + z_{s+t}),$$
for all $a \,\in\, \mathcal O_F$.

There is as well a natural multiplicative action $M$ of ${\mathcal O}^{*}_F$ on ${\mathbb C}^{s+t}$ given by 
the formula
$$M(u,\, (z_1,\, \cdots,\, z_{s+t}))\,=\,(\sigma_1(u)z_1,\, \cdots,\, \sigma_{s+t}(u)z_{s+t}),$$
for all $u\,\in\, \mathcal O_F^{*}$.

The above actions $M$ and $T$ together define an action of $\mathcal O_F^{*} \ltimes \mathcal O_F$ on 
${\mathbb C}^{s+t}$ by affine transformations. More precisely, an element $(u,\,a) \,\in\, \mathcal O_F^{*} 
\ltimes \mathcal O_F$ acts on ${\mathbb C}^{s+t}$ as $M_u \circ T_a$. The locus of fixed points of this 
$\mathcal O_F^{*} \ltimes \mathcal O_F$-action on ${\mathbb C}^{s+t}$ is contained in ${\mathbb R}^s \times 
{\mathbb C}^t$. Moreover, if one denotes by ${\mathcal O}_F^{*,+}$ the subgroup of positive units, meaning 
those units which are positive in all real embeddings of $F$, the above $\mathcal O_F^{*} \ltimes \mathcal 
O_F$ action induces an $\mathcal O_F^{*,+} \ltimes \mathcal O_F$ action on ${\mathbb C}^{s+t}$ preserving 
${\mathbb H}^s \times {\mathbb C}^t$, where $\mathbb H$ is the upper-half space. Therefore, this $\mathcal 
O_F^{*,+} \ltimes \mathcal O_F$--action on ${\mathbb H}^s \times {\mathbb C}^t$ is free. In general this 
action is not properly discontinuous, but it is proved in \cite{OT} that one can find for any $F$ an {\it 
admissible} subgroup $U \,\subset\, \mathcal O_F^{*,+}$ such that the induced action of $U \ltimes \mathcal 
O_F$ on ${\mathbb H}^s \times {\mathbb C}^t$ is properly discontinuous with compact quotient. Moreover, for 
$t\,=\,1$, any finite index subgroup $U \,\subset\, \mathcal O_F^{*,+}$ is admissible.

\begin{definition}
For any algebraic field extension $\lbrack F ,\, \mathbb Q \rbrack\, =\,n=\,s+2t$, with $s\,>\,0$ and $t\,>\,0$,
and any admissible subgroup $U \,\subset\, \mathcal O_F^{*,+}$, the compact complex manifold constructed
as quotient of ${\mathbb H}^s \times {\mathbb C}^t$ by the above $U \ltimes \mathcal O_F$--action
is called an {\it Oeljeklaus--Toma manifold} (abbreviated as {\it $\text{OT}$--manifold}).
\end{definition} 

It was proved in \cite{OT} that all Oeljeklaus--Toma manifolds are non-K\"ahler.

Since the above $U \ltimes \mathcal O_F$--action on ${\mathbb H}^s \times {\mathbb C}^t$ is by complex affine 
transformations, the affine structure of ${\mathbb H}^s \times {\mathbb C}^t$, induced by its inclusion in the 
standard complex affine space ${\mathbb C}^{s+t}$, descends on the quotient by $U \ltimes \mathcal O_F$ and 
endows any quotient $\text{OT}$--manifold with a standard {\it complex affine structure}.

A very useful result dealing with the algebraic dimension of the $\text{OT}$--manifolds was obtained in 
\cite{BO} where the authors prove that all $\text{OT}$--manifolds are of algebraic dimension zero. This means 
that there are no nonconstant meromorphic functions on a $\text{OT}$--manifold.

Therefore, $\text{OT}$--manifolds share with Inoue surfaces of class ${\rm VII}_0$ the properties that they are 
non-K\"ahler manifolds with algebraic dimension zero and admit a standard complex affine structure.

The {\it locally conformally K\"ahler geometry} of $\text{OT}$--manifolds was studied in \cite{OV, PV}. Let us 
recall that a {\it locally conformally K\"ahler metric} on a complex manifold $M$ is a conformal class of 
Hermitian metrics on $M$ which admits a local representative given by a K\"ahler metric. As it was first 
observed in \cite{Va}, $M$ admits a locally conformally K\"ahler metric if and only if $M$ has a Galois 
covering which admits a K\"ahler metric satisfying the condition that the Galois group of the covering
acts by holomorphic homotheties with respect to the K\"ahler metric.

In \cite{OT} its authors proved that for $t\,=\,1$ (number of complex non-real embeddings) all corresponding $\text{OT}$--manifolds 
admit a locally conformally K\"ahler metric.

Following \cite{OV}, we will say that a $\text{OT}$--manifold is locally conformally K\"ahler if it is
constructed from a number field $F$ with $t\,=\,1$ (meaning it admits exactly 2 complex non-real embeddings).

Ornea and Verbitsky proved in \cite{OV} that all locally conformally K\"ahler $\text{OT}$--manifolds are {\it 
simple}, meaning they do not admit closed proper subvarieties of positive dimension.

For general $\text{OT}$--manifolds, it was proved in \cite[Theorem 1.2]{OVV} that all their irreducible subvarieties $Z$ of smallest possible positive dimension are complex 
affine in the sense that at any smooth point of $Z$, the tangent sub-bundle to $Z$ is preserved by the 
standard flat affine connection of the ambient $\text{OT}$--manifold.

\section {Holomorphic geometric structures on Oeljeklaus--Toma manifolds}\label{sect struct OT}

\begin{theorem}\label{thm1}
Let $M$ be an $\text{OT}$--manifold, and let $\phi$ be a holomorphic geometric structure
of affine type on $M$. Then $\phi$ is locally homogeneous. Moreover, the Lie algebra of local holomorphic affine vector fields 
preserving $\phi$ is transitive on $M$.
\end{theorem}

\begin{proof}
We will prove that any $\text{OT}$--manifold $M$ satisfies the following property: Any non-zero holomorphic 
tensor has an empty vanishing set, namely for any integers $p,\,q \,\geq\, 0$ a non-zero holomorphic section of the 
holomorphic vector bundle $(TM)^{\otimes p}\otimes (T^*M)^{\otimes q}$ does not vanish
anywhere on $M$. Therefore, all 
holomorphic geometric structures of affine type on $M$ are locally homogeneous as a consequence of the 
criteria proved in Lemma 3.2 in \cite{D1}.
 
It was recently proved in \cite{Ka} (see Remark 3.2 of \cite{Ka}) that for any nonnegative integers $p,\,q$, the 
holomorphic vector bundle $(TM)^{\otimes p}\otimes (T^*M)^{\otimes q}$ over the $\text{OT}$--manifold splits as 
a sum of holomorphic line bundles over $M$. By Theorem 2.1 in \cite{APV}, any holomorphic line bundle over an 
$\text{OT}$--manifold is flat. Now a flat line bundle over an $\text{OT}$--manifold admits a non-zero holomorphic section if and only if it 
is trivial; consequently, the non-zero holomorphic section is a holomorphic section of the trivial bundle and 
it does not vanish anywhere. Indeed, this was first prove in \cite{In} (see Proposition 2, i) on page 274) for Inoue surfaces with the consequence that Inoue surfaces do
not contain curves. The same argument generalizes to $\text{OT}$--manifolds thanks to
Lemma 2.4 (on page 167) in \cite{OT}, with the consequence that $\text{OT}$--manifolds do not contain divisors (see Proposition 2.5 in \cite{APV} and Theorem 3.5 in \cite{BO}).

This finishes the proof of the above mentioned property, and, as mentioned
before Lemma 3.2 in \cite{D1} applies.

Notice that the juxtaposition of $\phi$ with the standard complex affine structure on $M$ provides a holomorphic geometric structure $\phi'$ on $M$ of affine type \cite{Gr,DG}. Since $\phi'$ is locally homogeneous, its Killing Lie algebra is transitive $M$. The Killing Lie algebra of $\phi'$ is the intersection of the Killing Lie algebra of $\phi$ with the Killing Lie algebra of the standard complex affine structure. It is the Lie algebra of affine vector fields preserving $\phi$; so this last Lie algebra is transitive on $M$.
\end{proof}

\begin{theorem}\label{lck OT}
Let $M$ be an $\text{OT}$--manifold which is locally conformally K\"ahler. Suppose $M$ is endowed with a 
holomorphic geometric structure $\phi$. Then $\phi$ is locally homogeneous. Moreover, the Lie algebra of local 
holomorphic affine vector fields preserving $\phi$ is transitive on $M$.
\end{theorem}

\begin{proof}
The starting point of the proof is the property, proved by Ornea and Verbitsky in \cite[Corollary 1.7]{OV}, 
that locally conformally K\"ahler $\text{OT}$--manifolds are of algebraic dimension zero, meaning they do not 
admit non-constant meromorphic functions. Therefore, Theorem 3 and its Corollary 4 in \cite{D1} (see also 
Theorem 2.1 and its Corollary 2.2 in \cite{D3}) imply that any holomorphic rigid geometric structures on $M$ 
is locally homogeneous on an open dense subset $M \setminus S$, where $S$ is a nowhere dense
closed analytic subset of $M$ of positive codimension.

Let us now consider a (possibly non-rigid) holomorphic geometric structure $\phi$ on $M$. Add together the 
holomorphic affine structure of the $\text{OT}$--manifold with $\phi$ in order to obtain a holomorphic rigid 
geometric structure in Gromov's sense (the juxtaposition of a geometric structure with a rigid one is 
a rigid geometric structure) \cite{Gr,DG}; denote the rigid geometric structure thus obtained by $\phi'$.
The Killing Lie algebra of $\phi'$ is the intersection of the Killing Lie 
algebra of $\phi$ with the sheaf of the affine vector fields (with respect to the affine structure of the 
$\text{OT}$--manifold). Applying the above mentioned result of \cite{D1}
to $\phi'$ one gets that the Killing Lie algebra of $\phi'$ 
is transitive on $M \setminus S$, with $S$ being the closed analytic subset of $M$ where the Killing
Lie algebra of $\phi'$ drops rank.

Moreover, the main result in \cite{OV} (Theorem 1.3 of \cite{OV}) proves that locally conformally K\"ahler 
$\text{OT}$--manifolds do not admit non-trivial complex submanifolds. This implies that $S$ is of dimension 
zero: it is a finite number of points.

To prove by contradiction, assume that $S$ is non-empty.

Let us consider the universal cover $\widetilde{M}$ of $M$ which identifies with the open set ${\mathbb H}^s \times {\mathbb C}^t$ in
the standard complex affine 
space $\mathbb C^{s+t}$. The pull-back of the sheaf of Killing fields of $\phi'$ to the universal cover 
$\widetilde{M}$ is generated by global section \cite{Am}: it identifies with the Lie subalgebra 
$\mathfrak{g'}$ of complex affine vector fields in $\mathfrak{gl}(s+t, \mathbb C) \ltimes {\mathbb C}^{s+t}$ which 
preserve the pull-back of $\phi$ to $\widetilde{M}$ (it is the Killing Lie algebra of the pull-back of 
$\phi'$ to $\widetilde{M}$). Let us consider a basis ($k_1,\, \cdots ,\,k_l)$ of $\mathfrak{g'}$. The inverse 
image of $S$ in universal cover is the subset $\widetilde{S} \,\subset\, \widetilde{M}$ which identifies with 
the common zeros of the affine vector fields $k_i$. Therefore $\widetilde{S}$ is an algebraically Zariski closed set in 
${\mathbb H}^s \times {\mathbb C}^t$ (it coincides with the intersection of ${\mathbb H}^s \times {\mathbb C}^t$ with a linear subvariety in ${\mathbb C}^{s+t}$). On the other hand we have seen that $\widetilde{S}$ has dimension zero. Moreover, since the 
fundamental group of $M$ is infinite, we know that $\widetilde S$ is an infinite discrete subset: it
cannot be an algebraically Zariski closed subset in ${\mathbb H}^s \times {\mathbb C}^t$; a contradiction. This completes the proof.
\end{proof}

In the case of a general $\text{OT}$--manifold the above proof gives the following weaker result:
 
\begin{proposition}
A holomorphic geometric structure $\phi$ on a $\text{OT}$--manifold is locally homogeneous on $M \setminus S$, 
with $S$ a nowhere dense analytic subset of positive codimension. Moreover, the Lie algebra of local holomorphic affine vector fields 
preserving $\phi$ is transitive on $M \setminus S$.
\end{proposition}

\begin{proof}
As in the proof of Theorem \ref{lck OT} one considers the holomorphic rigid geometric structure $\phi'$ which 
is the juxtaposition of the geometric structure $\phi$ with the standard complex affine structure of the 
$\text{OT}$--manifold $M$. The main result in \cite{BO} proves that all $\text{OT}$--manifolds have algebraic 
dimension zero. Therefore, Theorem 3 and its Corollary 4 in \cite{D1} (see also Theorem 2.1 and its Corollary 
2.2 in \cite{D3}) imply that $\phi'$ is locally homogeneous on an open dense subset $M \setminus S$, where $S$ is 
a closed subset in $M$. This implies that the Killing Lie algebra of $\phi'$, which coincides with the 
Lie algebra of affine vector fields preserving $\phi$, is transitive on $M \setminus S$. Moreover, $S$ 
coincides with the locus where this Lie algebra of affine vector fields drops rank. Consequently,
$S$ is a nowhere dense analytic subset of positive codimension in $M$.
\end{proof}

\begin{corollary}
Let $M$ be an $\text{OT}$--manifold associated to an admissible subgroup of units $U$ such that any element $u \in U \setminus \{1 \}$ 
is a primitive element for the number field $F$. Then any holomorphic geometric structure $\phi$ on $M$ is locally homogeneous. 
Moreover, the intersection of the Killing Lie algebra of $\phi$ with the Killing Lie algebra of the affine structure is transitive on 
$M$.
\end{corollary} 

\begin{proof}
The same proof of the previous Proposition, together with the fact that, under the assumptions, $M$ does not admit nontrivial complex 
subvarieties by \cite[Theorem 1.1]{OVV}, imply that the Killing Lie subalgebra of the affine structure preserving the holomorphic 
geometric structure is transitive away from a finite set of points $S$. We prove that $S$ is empty as in the proof of Theorem 
\ref{lck OT}.
\end{proof} 

Let us now state and prove the analogous results for Cartan geometries.

\begin{theorem}\label{alg dim 0} 
Let $M$ be a complex manifold of algebraic dimension zero, bearing a holomorphic torsion-free affine 
connection $\nabla$. Then the following statements hold:
\begin{enumerate}
\item[{(i)}] Any holomorphic Cartan geometry of algebraic type on $M$ is locally homogeneous on an open dense
subset $M 
\setminus S$, with $S$ a nowhere dense closed analytic subset of positive codimension.

\item[{(ii)}] The intersection of the Killing Lie algebra of the Cartan geometry with the Killing Lie algebra of 
$\nabla$ is transitive on $M \setminus S$.
\end{enumerate}
\end{theorem}

\begin{proof} Let $M$ be a complex manifold of complex dimension $n$.
Let $(E_H,\, \eta)$ be a Cartan geometry on $M$ of type $(G,\,H)$. There is a canonical
isomorphism of vector bundles between 
$TM$ and $E_H \times_{H} {\mathfrak g}/{\mathfrak h}$, where $H$ acts on ${\mathfrak g}/{\mathfrak h}$ via the 
adjoint representation ${\rm Ad}.$ Moreover, each point $p \,\in\, E_H$ lying above $b \,\in\, M$
defines an isomorphism $i_p \,:\, T_bM \,\longrightarrow\, {\mathfrak g}/{\mathfrak h}$ such that
$i_{ph}\,=\,{\rm Ad}(h^{-1}) i_p$, for all $h \,\in\, H$ (see 
\cite[p.~188, Theorem 3.15]{Sh}).

Denote by $K$ the kernel the adjoint representation ${\rm Ad} \,:\, H \,\longrightarrow\, {\rm GL}({\mathfrak 
g}/ {\mathfrak h})$, and also denote by $F_H$ the quotient of $E_H$ by $K$. It follows that the principal 
$H/K$--bundle $F_H$ get identified with a principal subbundle of the first jet bundle $R^1(M)$. This defines a 
quotient map $$i\,\,: \,\,E_H \,\,\longrightarrow\,\, F_H \,\,\simeq\,\, R^1(M).$$

Notice that the form $\eta$ trivializes the holomorphic tangent vector bundle of $E_H$ and gives an
isomorphism of vector bundles: $TE_H \,\simeq\, E_H \times {\mathfrak g}$. The curvature $\Omega$ of
the Cartan connection is completely
determined by 
an $H$-equivariant map $\Omega \,:\, E_H \,\longrightarrow\, V$, with $V\,=\,\Lambda^2({\mathfrak g}/{\mathfrak h}) \otimes 
{\mathfrak g}$, with $H$ acting linearly on $V$ by
$$ h \cdot l(u,\,v)\,=\,({\rm Ad} (p) \circ l) ({\widetilde{\rm Ad}} (h^{-1})u,\, {\widetilde{\rm Ad}} (h^{-1})v)$$
for all $h \,\in\, H$, where ${\widetilde{\rm Ad}}$ is the induced adjoint
$H$-action on ${\mathfrak g}/{\mathfrak h}$ (see \cite[Chapter 5, Lemma 3.23]{Sh}).

Following \cite{Me, Pe} we define the $m$--jet of the Cartan curvature $\Omega$ of $(E_H,\, \eta)$ as a map:
$$J^{m} \Omega \,\,:\,\, E_H \,\,\longrightarrow\,\,{\rm Hom} (\otimes^m \mathfrak g, V)$$ defined as 
$J^m \Omega (b h^{-1})\,=\,h \cdot (J^m \Omega (b) \circ {\rm Ad}^mh^{-1})$, for all $b \,\in\, E_H$,
with ${\rm Ad}^m$ being the $m$-th tensor power $\bigotimes ^m \mathfrak g$ of the adjoint representation
restricted of $H$ (see \cite[Proposition 3.1]{Me}).

Consider now, as in \cite{Me, Pe}, the map $${\mathcal D}^m \,:\, E_H \,\longrightarrow\,
V \oplus {\rm Hom}(\mathfrak g,\, V) \oplus \cdots\oplus {\rm Hom} (\otimes^m \mathfrak g,\, V)$$
defined as ${\mathcal D}^m \,=\, (\Omega,\, J^1 \Omega,\, \cdots, \,J^m \Omega).$ 

It was proved in \cite[Theorem 4.1]{Me} and in \cite[Theorem 1.1] {Pe} that, for
$m\,=\,{\rm dim (H)}$, the image of the map ${\mathcal 
D}^m$ is exactly a $H$-orbit if and only if the Cartan geometry
$(E_H, \,\eta)$ is locally homogeneous.

Notice that $(E_H,\, \eta)$ being of algebraic type the $H$--action on
$$V \oplus {\rm Hom}({\mathfrak g},\, V) \oplus \cdots\oplus {\rm Hom} (\otimes^m{\mathfrak g},\, V)$$ is
algebraic. For algebraic actions, 
Rosenlicht's Theorem, \cite{Ro}, implies that $H$-orbits in general position in $V \oplus {\rm Hom}({\mathfrak g},\, V)
\oplus \cdots\oplus {\rm Hom} (\otimes^m {\mathfrak g},\, V)$ are separated by rational functions. In other
words, any $H$-invariant rational function $\mathcal R$ on 
$V \oplus {\rm Hom}({\mathfrak g},\, V) \oplus \cdots \oplus {\rm Hom} (\otimes^m {\mathfrak g},\, V)$
defines an $H$-invariant meromorphic map ${\mathcal R} \circ {\mathcal D}^m$ on $E_H$. Consequently, 
the meromorphic map ${\mathcal R} \circ {\mathcal D}^m$ descends on $M$ as a constant function, since
the algebraic dimension of $M$ is zero. This implies that $\mathcal R$ is constant on
${\mathcal D}^m(E_H)$. From this it follows that there exists an open dense
set $U$ in $M$ such that ${\mathcal D}^m(E_H \arrowvert_U)$ is exactly an $H$-orbit. Therefore, $(E_H, \,
\eta)$ is locally homogeneous over $U$; more details can be found in \cite[Corollary 1.3]{D2}.

Since the Killing Lie algebra for $(E_H,\, \eta)$ is transitive on an open dense subset, it follows that that $U
\,=\, M\setminus S$, with $S$ being the closed subset of $M$ where the Killing Lie algebra drops rank.
Consequently, $S$ is a 
nowhere dense analytic subset of positive codimension in $M$. This proves point (i) of Theorem \ref{alg dim 
0}.

Proof of (ii). Let us now consider the geometric structure which is the juxtaposition of the Cartan geometry 
$(E_H, \,\eta)$ and the holomorphic affine connection $\nabla$. This juxtaposition is not in general a Cartan 
geometry, but we will see that this can be handled as an extended Cartan geometry in the sense of \cite{Pe}.

Consider a point $(b,\,l)$ in the frame bundle $R^1(M)$ above $b \,\in\, M$. This point defines an isomorphism 
$l$ between ${\mathbb C}^n$ and $T_bM$. The point $(b,\,l) \,\in\, R^1(M)$ defines a unique system of local 
coordinates at $b \,\in\, M$ which are exponential with respect to $\nabla$ (meaning that the Christoffel 
coefficients of $\nabla$ vanishes at the origin for this system of local coordinates). Then we consider the 
$m$-jet of the curvature of $\nabla$ at the origin with respect to this exponential coordinates. This gives us 
a ${\rm GL}(n, \mathbb C)$-equivariant map $R^1(M)\,\longrightarrow\, W^m$, with $W^m$ being the vector space 
of $m$-jets of curvature. Replacing the point $(b,\,l) \,\in\, R^1(M)$ by $(b, \,l')\,\in\, R^1(M)$, where 
$l'\,=\,lh$ with $h \,\in\, {\rm GL}(n, \mathbb C)$, gives another system of local exponential coordinates at 
$b$, and taking the $m$-jet of the curvature in those coordinates defines the action of ${\rm GL}(n, \mathbb 
C)$ on $W^m$. This action is algebraic (see \cite{DG} for more details).

Therefore the $m$-jet of the curvature of $\nabla$ is given by a ${\rm GL}(n, \mathbb C)$-equivariant map 
$${\mathcal K}^m: R^1(M) \longrightarrow W^m.$$
Now consider the $H$-equivariant map
$$E_H \,\longrightarrow\, V \oplus {\rm Hom}({\mathfrak g},\, V) \oplus \ldots
\oplus {\rm Hom} (\otimes^m {\mathfrak g}, V) \oplus W^m$$ defined by 
$$ p\,\, \longmapsto\,\, {\mathcal D}^m(p) \oplus ({\mathcal K}^m \circ i )(p),$$
where $i\,:\, E_H \,\longrightarrow \,F_H\, \simeq\, R^1(M)$ is the quotient bundle map described above.

As before, Rosenlicht's Theorem, \cite{Ro}, implies that there exists an open dense set $U$ in $M$ such that $({\mathcal 
D}^m \oplus ({\mathcal K}^m \circ i ))(E_H\arrowvert_U)$ is exactly an $H$-orbit. Therefore, $(E_H, \,\eta)$ 
together with $\nabla$ form an extended Cartan geometry which is locally homogeneous over $U$ \cite[Section 
4.4.2 and Section 4.4.3]{Pe}. Hence the intersection between the Killing Lie algebra of $(E_H,\, \eta)$ and 
the Killing Lie algebra of $\nabla$ is transitive on $U$. It follows that that $U\,=\, M \setminus S$, with $S$ 
being the closed subset of $M$ where the Killing Lie algebra drops rank. Therefore, $S$ is a nowhere
dense analytic subset of positive codimension in $M$.
\end{proof}

\begin{corollary}\label{Cor 1}
A holomorphic Cartan geometry of algebraic type on an $\text{OT}$--manifold is locally homogeneous
on $M \setminus S$, where $S$ is a nowhere dense analytic subset of positive codimension in $M$. Moreover, the 
Lie algebra of local holomorphic affine vector fields on $M$ preserving the Cartan geometry is transitive on $M \setminus S$.
\end{corollary} 

\begin{proof}
Let $(E_H, \,\eta)$ be a holomorphic Cartan geometry of algebraic type on the $\text{OT}$--manifold $M$. Since 
$\text{OT}$--manifolds are known to be of algebraic dimension zero \cite{BO} and are endowed, by construction, 
with a standard affine structure, \cite{OT}, Theorem \ref{alg dim 0} applies to $(E_H,\, \eta)$ and to the flat 
torsion-free affine connection $\nabla$ associated to the affine structure. Hence the Lie algebra of 
affine vector fields preserving the Cartan geometry $(E_H,\, \eta)$ is transitive on $M \setminus S$. Since $S$ 
is the subset in $M$ where the Lie algebra of affine vector fields preserving $(E_H, \,\eta)$ drops rank, it
follows that $S$ is a nowhere dense analytic subset of positive codimension in $M$.
\end{proof}

\begin{corollary}\label{Cartan lck OT}
A holomorphic Cartan geometry of algebraic type on a locally conformally K\"ahler $\text{OT}$--manifold is 
locally homogeneous. Moreover, the intersection of the Killing Lie algebra of the Cartan geometry with the 
Killing Lie algebra of the affine structure is transitive.
\end{corollary}

\begin{proof}
By Corollary \ref{Cor 1}, the intersection of the Killing Lie algebra of the Cartan geometry with the 
Killing Lie algebra of the affine structure is transitive away from the analytic subset $S$. By \cite[Theorem 3.1]{OV}, locally
conformally K\"ahler $\text{OT}$--manifolds do not admit 
nontrivial complex submanifolds. Therefore $S$ is a finite union of points. We conclude as in Theorem \ref{lck 
OT} that $S$ is the empty set.
\end{proof} 

\begin{corollary}
Let $M$ be an $\text{OT}$--manifold associated to an admissible subgroup of units $U$ such that any element $u \in U \setminus \{1 \}$ 
is a primitive element for the number field $F$. Then any holomorphic Cartan geometry of algebraic type on $M$ is locally 
homogeneous. Moreover, the intersection of the Killing Lie algebra of the Cartan geometry with the Killing Lie algebra of the affine 
structure is transitive.
\end{corollary} 

\begin{proof} 
By Corollary \ref{Cor 1}, the intersection of the Killing Lie algebra of the Cartan geometry with the Killing Lie algebra of the 
affine structure is transitive away from the analytic subset of positive codimension $S$. By \cite[Theorem 1.1]{OVV} $M$ does not 
admit nontrivial complex submanifolds. Therefore $S$ is a finite union of points. We conclude as in Theorem \ref{lck OT} that $S$ is 
the empty set.
\end{proof}

\section*{Acknowledgements}

We thank the referee for showing how to correct an argument in the proof of Theorem \ref{thm 1 introduction}.
The authors are grateful to Vincent Pecastaing for helpful and friendly discussions and, in particular, for his suggestion to use his results about symmetries of extended Cartan 
geometries, \cite{Pe}, in the proof of Theorem \ref{alg dim 0} and its Corollary \ref{Cartan lck OT}. The authors are also grateful to Hisashi Kasuya,
Liviu Ornea, Misha Verbitsky and Victor Vuletescu for helpful comments. The first author is partially 
supported by a J. C. Bose Fellowship (JBR/2023/000003).

Declaration of interests. The authors do not work for, advise, own shares in, or receive
 funds from any organisation that could benefit from this article, and have declared no
affiliation other than their research organisations.

Data availability statement. No data were generated or used.

\end{document}